\documentclass[12pt]{amsart}

\usepackage{amsmath,mathtools,amsthm,amsfonts,amssymb,color,verbatim,graphicx,cases}
\usepackage[letterpaper]{geometry}
\usepackage{stmaryrd,url}
\usepackage{enumitem}
\setlistdepth{5}

\newlist{myEnumerate}{enumerate}{9}
\setlist[myEnumerate,1]{label=(\Roman*)}
\setlist[myEnumerate,2]{label=(\arabic*)}
\setlist[myEnumerate,3]{label=(\Alph*)}
\setlist[myEnumerate,4]{label=(\roman*)}
\setlist[myEnumerate,5]{label=(\alph*)}

\usepackage{tikz}
\usetikzlibrary{shapes.geometric}
\usetikzlibrary{positioning}
\usetikzlibrary{calc}
\usetikzlibrary{scopes}
\usetikzlibrary{shapes}
\usetikzlibrary{decorations.pathmorphing}

\usepackage{pgf}
\usetikzlibrary{spy,backgrounds,decorations,positioning, decorations.markings, arrows,shapes, calc}

\usepackage{color}

\theoremstyle{definition}
\newtheorem{theorem}{Theorem}[section]
\newtheorem{corollary}[theorem]{Corollary}

\newtheorem{lemma}[theorem]{Lemma}
\newtheorem{proposition}[theorem]{Proposition}

\newtheorem{definition}[theorem]{Definition}

\newtheorem{rem}[theorem]{Remark}
\newtheorem{exam}[theorem]{Example}

\definecolor{mjo}{rgb}{0,0,.9}

\global\long\def\gen{\text{\sf GEN}}

\global\long\def\mex{\operatorname{mex}}
\global\long\def\nim{\operatorname{nim}}
\global\long\def\opt{\operatorname{Opt}}

\global\long\def\type{\operatorname{type}}
\global\long\def\otype{\operatorname{otype}}

\global\long\def\spr{d}

\begin{document}
\setlength{\jot}{0pt} 

\title[Impartial achievement games for generating nilpotent groups]
{Impartial achievement games for generating\\  nilpotent groups}

\author{Bret J.~Benesh}
\address{
Department of Mathematics,
College of Saint Benedict and Saint John's University,
37 College Avenue South,
Saint Joseph, MN 56374-5011, USA
}
\email{bbenesh@csbsju.edu}
\author{Dana C.~Ernst}
\author{N\'andor Sieben}
\address{
Department of Mathematics and Statistics,
Northern Arizona University PO Box 5717,
Flagstaff, AZ 86011-5717, USA
}
\email{Dana.Ernst@nau.edu, Nandor.Sieben@nau.edu}

\date{\today}


\begin{abstract}
We study an impartial game introduced by Anderson and Harary.  The game is played by two players who alternately choose previously-unselected elements of a finite group. The first player who builds a generating set from the jointly-selected elements wins. 
We determine the nim-numbers of this game for finite groups of the form $T \times H$, where $T$ is a $2$-group and $H$ is a group of odd order. This includes all nilpotent and hence abelian groups.
\end{abstract}

\maketitle


\section{Introduction}


Anderson and Harary~\cite{anderson.harary:achievement} introduced an impartial combinatorial game in which two players alternately take turns selecting previously-unselected elements of a finite group $G$ until the group is generated by the jointly-selected elements. The first player who builds a generating set from the jointly-selected elements wins this achievement game denoted by $\gen(G)$. The outcome of $\gen(G)$ was determined for finite abelian groups in~\cite{anderson.harary:achievement}. In~\cite{Barnes}, Barnes provides criteria for determining the outcome for an arbitrary finite group, and he applies his criteria to determine the outcome of some of the more familiar finite groups, including cyclic, abelian, dihedral, symmetric, and alternating groups.

A fundamental problem in game theory is to determine nim-numbers of impartial two-player games.  The nim-number allows for the easy calculation of the outcome of the sum of games.  A general theory of impartial games appears in~\cite{albert2007lessons,SiegelBook}. A framework for computing nim-numbers for $\gen(G)$ is developed in~\cite{ErnstSieben}, and the authors determine the nim-numbers for $\gen(G)$ when $G$ is a cyclic, abelian, or dihedral group. The nim-numbers for symmetric and alternating groups are determined in~\cite{BeneshErnstSiebenSymAlt} while generalized dihedral groups are addressed in~\cite{BeneshErnstSiebenGeneralizedDihedral}.

The task in this paper is to determine the nim-numbers of $\gen(G)$ for groups of the form $G=T\times H$ where $T$ is a finite $2$-group and $H$ is a group of odd order.  These groups have a Sylow $2$-direct factor.  Finite nilpotent groups are precisely the groups that can be written as a direct product of their Sylow subgroups, so the class of groups with a Sylow $2$-direct factor contains the nilpotent groups.  Note that groups with a Sylow $2$-direct factor are necessarily solvable by the Feit--Thompson Theorem~\cite{FeitThompson}.

Anderson and Harary~\cite{anderson.harary:achievement} also introduced a related avoidance game in which the player who cannot avoid building a generating set loses. As in the case of the achievement game, Barnes~\cite{Barnes} determines the outcome for a few standard families of groups, as well as a general condition to determine the player with the winning strategy.  The determination of the nim-numbers for the avoidance game for several families of groups appears in~\cite{BeneshErnstSiebenSymAlt,BeneshErnstSiebenDNG,ErnstSieben}. Similar algebraic games are studied by Brandenburg in \cite{brandenburg:algebraicGames}.

\section{Preliminaries}\label{section:preliminaries}


We now give a more precise description of the achievement game $\gen(G)$ played on a finite group $G$.   We also recall some definitions and results from~\cite{ErnstSieben}. In this paper, the cyclic group of order $n$ is denoted by $\mathbb{Z}_n$.  Other notation used throughout the paper is standard such as in~\cite{IsaacsFiniteGroups}. 
The nonterminal positions of $\gen(G)$ are exactly the nongenerating subsets of $G$. A terminal position is a generating set $S$ of $G$ such that there is a $g \in S$ satisfying $\langle S \setminus \{g\}\rangle < G$.  
The starting position is the empty set since neither player has chosen an element yet.
The first player chooses $x_{1}\in G$, and the designated player selects $x_{k}\in G\setminus\{x_{1},\ldots,x_{k-1}\}$ at the $k$th turn.
A position $Q$ is an option of $P$ if $Q=P \cup \{g\}$ for some $g \in G \setminus P$. The set of options of $P$ is denoted by $\opt(P)$. The player who builds a generating set from the jointly-selected elements wins the game. 

It is well-known that the second player has a winning strategy if and only if the nim-number of the game is $0$.  The only position of $\gen(G)$ for a trivial $G$ is the empty set, and so the second player wins before the first player can make a move.  Thus, $\gen(G)=*0$ if $G$ is trivial.  For this reason, we will assume that $G$ is nontrivial for the remainder of this section, and we will not need to consider trivial groups until Section~\ref{sec:final}.

The set $\mathcal{M}$ of maximal subgroups play a significant role in the game.  The last two authors define in~\cite{ErnstSieben} the set 
\[ 
\mathcal{I}:=\{\cap\mathcal{N}\mid\emptyset\not=\mathcal{N\subseteq\mathcal{M}}\} 
\]
of \emph{intersection subgroups}, which is the set of all possible intersections of maximal subgroups. We also define $\mathcal{J}:=\mathcal{I}\cup\{G\}$. The smallest intersection subgroup is the Frattini subgroup $\Phi(G)$ of $G$.

For any position $P$ of $\gen(G)$ let 
$$
\lceil P \rceil:=\bigcap \{I\in \mathcal{J} \mid P\subseteq I \}
$$ 
be the the smallest element of $\mathcal{J}$ containing $P$. We write $\lceil P,g_1,\ldots,g_n\rceil$ for $\lceil P \cup \{g_1,\ldots,g_n\}\rceil$ and $\lceil g_1,\ldots,g_n \rceil$ for $\lceil \{g_1,\ldots,g_n\}\rceil$ if $g_1,\ldots,g_n \in G$ .

Two positions $P$ and $Q$ are \emph{structure equivalent} if $\lceil P \rceil=\lceil Q \rceil$. The \emph{structure class} $X_I$ of $I\in\mathcal{J}$ is the equivalence class of $I$ under this equivalence relation.  Note that the definitions of $\lceil P \rceil$ and $X_I$ differ from those given in~\cite{BeneshErnstSiebenSymAlt,BeneshErnstSiebenDNG,BeneshErnstSiebenGeneralizedDihedral,ErnstSieben}, but it is easy to see that these definitions are equivalent to the originals.  
We let $\mathcal{Y}:=\{X_{I}\mid I\in\mathcal{J}\}$. We say $X_J$ is an option of $X_I$ if $Q\in\opt(P)$ for some $P\in X_I$ and $Q\in X_J$. The set of options of $X_I$ is denoted by $\opt(X_I)$.

The \emph{type} of the structure class $X_{I}$ is the triple 
\[
\type(X_{I}):=(|I| \text{ mod } 2,\nim(P),\nim(Q)),
\] where $P,Q\in X_{I}$ with $|P|$ even and $|Q|$ odd.  This is well-defined by~\cite[Proposition~4.4]{ErnstSieben}.        We define the \emph{option type} of $X_I$ to be the set 
\[
\otype(X_I):=\{\type(X_J) \mid X_J \in \opt(X_I)\}.
\]
We say the parity of $X_I$ is the parity of $|I|$.

The nim-number of the game  is the  nim-number 
of the initial position $\emptyset$, which is an even-sized subset of $\Phi(G)$. 
Because of this, $\nim(\gen(G))$ is the second component of 
\[
\type(X_{\Phi(G)}) =(|\Phi(G)| \text{ mod } 2,\nim(\emptyset),\nim(\{e\})).
\] 

 We use the following result of \cite{ErnstSieben} as our main tool to compute nim-numbers.  Note that $\type(X_G)=(|G|\!\mod 2,0,0)$. Recall that $\mex(A)$ for a subset $A \subseteq \mathbb{N} \cup\{0\}$ is the least nonnegative integer not in $A$.

\begin{proposition} 
For $X_I\in\mathcal{Y}$ define
\[ 
A_I=\{ a \mid (\epsilon, a, b)\in \otype(X_I) )\},\qquad
B_I=\{b \mid (\epsilon, a, b)\in \otype(X_I) )\}.
\]
Then $\type(X_I)=(|I| \!\mod 2,a,b)$ where
\begin{gather*}
a:=\mex(B_I),\  b:=\mex(A_I\cup\{a\}) \text{ if $|I|$ is even} \\
b:=\mex(A_I),\  a:=\mex(B_I\cup\{b\}) \text{ if $|I|$ is odd}. \\
\end{gather*}
\end{proposition}

\begin{figure}
\includegraphics{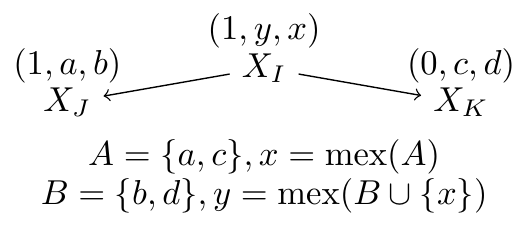}
\caption{\label{fig:type} Example of a calculation for $\type(X_I)$ if $\opt(X_I)=\{X_J,X_K\}$ where $X_I$ and $X_J$ are odd and $X_K$ is even.  The ordered triples are the types of the structure classes.}
\end{figure}

The previous proposition implies that the type of a structure class $X_I$ is determined by the parity of $X_I$ and the types of the options of $X_I$.  Figure~\ref{fig:type} shows an example of this calculation when $X_I$ is odd.


\section{Deficiency}


We will develop some general tools in this section.  For a finite group $G$, the minimum size of a generating set is denoted by \[\spr(G) :=  \min\{|S| : \langle S \rangle=G\}.\]   The following definition, which first appeared in~\cite{BeneshErnstSiebenGeneralizedDihedral}, is closely related to $\spr(G)$.

\begin{definition}
The \emph{deficiency} of a subset $P$ of a finite group $G$ is the minimum size $\delta_G(P)$ of a subset $Q$ of $G$ such that $\langle P\cup Q\rangle=G$.   For a structure class $X_I$ of $G$, we define $\delta_G(X_I)$ to be $\delta_G(I)$.
\end{definition}

Note that $P \subseteq Q$ implies $\delta_G(P) \geq \delta_G(Q)$.  

\begin{proposition}\label{prop:deltaofsubsets}
If $S \in X_I$, then $\delta_G(S)=\delta_G(I)$.
\end{proposition}

\begin{proof}
Let $n:=\delta_G(I)$ and $m:=\delta_G(S)$.  Since $S\subseteq I$, it follows as mentioned above that $n \leq m$. Now let $h_1,\ldots,h_n \in G$ such that $\langle I, h_1,\ldots,h_n\rangle=G$.  For a maximal subgroup $M$,  $I \subseteq M$ if and only if $S \subseteq M$ since $S \in X_I$.  Then since $\langle I, h_1,\ldots,h_n\rangle$ is not contained in any maximal subgroup, we conclude that neither is $\langle S, h_1,\ldots,h_n\rangle$.  Thus, $\langle S, h_1,\ldots,h_n\rangle=G$ and $\delta_G(S) \leq \delta_G(I)$, so $\delta_G(S)=\delta_G(I)$. 
\end{proof}

\begin{corollary}\label{cor:deltaofsubsets}
The deficiency of a generating set of a finite group $G$ is $0$ and $\delta_G(\emptyset)=\delta_G(\Phi(G))=d(G)$.
\end{corollary}

\begin{definition}
Let $G$ be a finite group, $\mathcal{E}$ be the set of even structure classes, and $\mathcal{O}$ be the set of odd structure classes in $\mathcal{Y}$.
We define the following sets:
\[
\begin{aligned}
\mathcal{D}_m &:= \{ X_I\in \mathcal{Y} \mid \delta_G(I)=m \}, & \mathcal{D}_{\geq m} &:= \bigcup \{\mathcal{D}_k \mid k \geq m\} \\ 
\mathcal{E}_m &:= \mathcal{E}\cap \mathcal{D}_m, & \mathcal{E}_{\geq m} &:= \bigcup \{\mathcal{E}_k \mid k \geq m\}\\ 
\mathcal{O}_m &:= \mathcal{O}\cap \mathcal{D}_m,  & \mathcal{O}_{\geq m} & := \bigcup \{\mathcal{O}_k \mid k \geq m\} \\  
\end{aligned}
\]
\end{definition}

\begin{proposition}\cite[Proposition~3.8 and Corollary~3.9]{BeneshErnstSiebenGeneralizedDihedral}\label{prop:setoptions}
Let $G$ be a finite group and $m$ be a positive integer. 
If $X_I \in \mathcal{D}_m$, then $X_I$ has an option in $\mathcal{D}_{m-1}$, and every option of $X_I$ is in $\mathcal{D}_m \cup \mathcal{D}_{m-1}$.
Moreover, if $X_I \in \mathcal{E}_m$, then $X_I$ has an option in $\mathcal{E}_{m-1}$, and every option of $X_I$ is in $\mathcal{E}_m \cup \mathcal{E}_{m-1}$.
\end{proposition}

Note that $\mathcal{D}_0=\{X_G\}$. Also, Proposition~\ref{prop:setoptions} implies that $\nim(P) \not=0$ for all $X_{\lceil P \rceil} \in \mathcal{D}_{1}$.  In the next lemma, we will use $\pi_i$ to denote the projection of a direct product to its $i$th factor.

\begin{lemma}\label{lem:FactorDeficiencies}
If $G$ and $H$ are finite groups and $S \subseteq G \times H$, then 
\[
\delta_{G \times H}(S) \geq\max\{\delta_{G}(\pi_1(S)),\delta_{H}(\pi_2(S))\}.
\]
\end{lemma}

\begin{proof}
Let $(x_1,y_1),\ldots,(x_k,y_k) \in G \times H$ be such that $\langle S,(x_1,y_1),\ldots,(x_k,y_k) \rangle = G \times H$.  Then $\langle \pi_1(S),x_1,\ldots,x_k\rangle = G$ and $\langle \pi_2(S),y_1,\ldots,y_k\rangle = H$, which yields the desired result. 
\end{proof}

\begin{lemma}\label{lem:FullFactors}
If $G$ and $H$ are finite groups and $S \subseteq G$, then $\delta_{G\times H}(S \times H) = \delta_G(S)$.
\end{lemma}
\begin{proof}
By Lemma~\ref{lem:FactorDeficiencies}, we have $\delta_{G \times H}(S \times H) \geq \delta_G(S)$.   Now let $n:=\delta_G(S)$.  Then there exist $g_1,\ldots,g_n \in G$ such that $\langle S,g_1,\ldots,g_n \rangle=G$.  Then $\langle S \times H, (g_1,e),\ldots,(g_n,e) \rangle = G\times H$.  Thus, $\delta_{G}(S) \geq \delta_{G\times H}(S \times H)$.  
\end{proof}

\begin{lemma}
If $G$ and $H$ are finite groups, then 
\[
\max\{d(G),d(H)\} \leq d(G \times H) \leq d(G)+d(H).
\]
\end{lemma}

\begin{proof}
We have $d(G)=\delta_G(\{e\})$, so for $K \in \{G,H\}$ \[d(G \times H)=\delta_{G\times H}(\{e\} \times \{e\}) \geq \delta_K(\{e\}) = d(K)\] by Lemma~\ref{lem:FactorDeficiencies}. Hence $\max\{d(G),d(H)\} \leq d(G \times H)$.  Let $n=d(G)$ and $m=d(H)$, and let $g_1,\ldots,g_n \in G$ such that $\langle g_1,\ldots,g_n \rangle =G$ and $h_1,\ldots,h_m \in H$ such that $\langle h_1,\ldots,h_m \rangle =H$.  Then $\langle (g_1,e),\ldots,(g_n,e),(e,h_1),\ldots,(e,h_m)\rangle = G \times H$, so 
\[
d(G \times H) \leq d(G)+d(H).
\qedhere
\] 
\end{proof}


\section{The Achievement Game $\gen(T \times H)$}\label{sec:final}


We now determine the nim-number of $\gen(T \times H)$ where $T$ is a finite $2$-group and $H$ has odd order.  We will split the analysis into different cases according to the parity of $|T \times H|$ and the value of $d(T \times H)$.  

If $T$ is trivial, then $T \times H \cong H$ and we can apply the following refinement of~\cite[Corollary~4.8]{ErnstSieben}.

\begin{proposition}\label{prop:OrderGIsOdd}
If $|H|$ is odd, then
\[
	   \gen(H) = 
	     \begin{dcases}
               *0, & \text{if } |H|=1\\
	       *2, & \text{if } |H|>1 \text{ and } d(H) \in \{1,2\}\\
	       *1, & \text{otherwise}.
	     \end{dcases}
\] 
\end{proposition}
\begin{proof}
The case where $|H|=1$ was done in Section~\ref{section:preliminaries}.  We proceed by structural induction on the structure classes to show that
\[
           \type(X_I) =
             \begin{dcases}
               (1,0,0), &  \text{if } X_I \in \mathcal{O}_0\\
               (1,2,1), &  \text{if } X_I \in \mathcal{O}_1\\
               (1,2,0), &  \text{if } X_I \in \mathcal{O}_2\\
               (1,1,0), &  \text{if } X_I \in \mathcal{O}_{\geq 3}.
             \end{dcases}
\]

\noindent
Every structure class in $\mathcal{O}_0$ is terminal, so $\type(X_I)=(1,0,0)$ if $X_I \in \mathcal{O}_0$.  If $X_I\in\mathcal{O}_1$, then $\{(1,0,0)\}\subseteq\otype(X_I)\subseteq\{(1,0,0),(1,2,1)\}$ by induction and Proposition~\ref{prop:setoptions}, which implies $\type(X_I)=(1,2,1)$.  Similarly, if $X_I \in \mathcal{O}_2$, then $\{(1,2,1)\}\subseteq\otype(X_I)\subseteq\{(1,2,0),(1,2,1)\}$, and so $\type(X_I)=(1,2,0)$. Again, if $X_I \in \mathcal{O}_3$, then $\{(1,2,0)\}\subseteq\otype(X_I)\subseteq\{(1,1,0),(1,2,0))\}$, and hence $\type(X_I)=(1,1,0)$.  Now if $X_I \in \mathcal{O}_{\geq 4}$, then $\otype(X_I)=\{(1,1,0)\}$ by induction, so $\type(X_I)=(1,1,0)$.  
 
Since $X_{\Phi(H)} \in \mathcal{O}_{d(H)}$ by~\cite[Proposition~3.7]{BeneshErnstSiebenGeneralizedDihedral}, the result follows from the fact that $\gen(H)$ equals the second component of $\type(X_{\Phi(H)})$.
\end{proof}

If $T$ is nontrivial, then we handle four cases in increasing complexity:  $d(T \times H)=1$, $d(T \times H) \geq 4$, $d(T \times H)=3$, and $d(T \times H)=2$.

\begin{proposition}\cite[Corollary~6.9]{ErnstSieben}
\label{prop:dIs1}
If $T$ is a nontrivial $2$-group and $H$ is a group of odd order such that $d(T \times H)=1$, then
\[
\gen(T \times H)=
\begin{dcases}
*1,          &  T \times H \cong \mathbb{Z}_{4k} \text{ for some } k \geq 1 \\
*2,          & T \times H \cong \mathbb{Z}_{2}  \\
*4,         & T \times H \cong \mathbb{Z}_{4k+2} \text{ for some } k \geq 1 .
\end{dcases}
\]
\end{proposition}

\begin{proposition}\cite[Corollary~3.11]{BeneshErnstSiebenGeneralizedDihedral}
\label{prop:dIsLarge}
If $|G|$ is even and $d(G) \geq 4$, then $\gen(G)=*0$. 
\end{proposition}

The following result will be useful in the case where $d(T \times H) \geq 2$.

\begin{proposition}\cite[Proposition~3.10]{BeneshErnstSiebenGeneralizedDihedral}\label{prop:GENeventypes}
If $G$ is a group of even order, then
\[
           \type(X_I) =
             \begin{dcases}
               (0,0,0), &  X_I \in \mathcal{E}_0\\
               (0,1,2), &  X_I \in \mathcal{E}_1\\
               (0,0,2), &  X_I \in \mathcal{E}_2\\
               (0,0,1), &  X_I \in \mathcal{E}_{\geq 3}.
             \end{dcases}
\]
\end{proposition}

\begin{proposition}\label{prop:dgt3}
If $T$ is a nontrivial $2$-group and $H$ is a group of odd order such that $d(T \times H)= 3$, then $\gen(T \times H)=*0$.
\end{proposition}

\begin{proof}
Let $g$ be the element the first player initially selects, so the game position is $\{g\} \in X_{\lceil g \rceil}$. If $X_{\lceil g \rceil}\in\mathcal{E}_{\ge 2}$, then the second player selects the identity $e$ and keeps the resulting game position $\{g,e\}$ in $X_{\lceil g,e \rceil}=X_{\lceil g \rceil}$. 

Otherwise, $X_{\lceil g \rceil}\in\mathcal{O}_{\ge 2}$, so $g$ has odd order and can be written as $g=(e,h)$ for some $h \in H$.  In this case, the second player selects $(t,e)$ for some involution $t \in T$.  Then the resulting position $\{(e,h), (t,e)\}$ is in $X_{\lceil (e,h),(t,e) \rceil}=X_{\lceil (t,h) \rceil} \in \mathcal{E}_{\ge 2}$. 

In both cases the position after the second move has nim-number $0$ since it is in a structure class with type $(0,0,2)$ or $(0,0,1)$ by Proposition~\ref{prop:GENeventypes}. Thus, the second player wins. 
\end{proof}

Lastly, we consider the case where $d(T \times H)=2$. First, we handle the subcase when $\Phi(T)$ is nontrivial.  

\begin{proposition}
\label{PhiTriv}
If $T$ is a $2$-group and $H$ is a group of odd order such that $d(T \times H)=2$ and $\Phi(T)$ is nontrivial, then $\gen(T \times H)=*0$.
\end{proposition}

\begin{proof}
Because $\Phi(T \times H)\cong \Phi(T) \times \Phi(H)$ by~\cite[Theorem~2]{dlab1960}, we conclude that the order of $\Phi(T \times H)$ is even.  Since $d(T \times H)=2$, we have $X_{\Phi(T \times H)} \in \mathcal{E}_2$, so $\type(X_{\Phi(T \times H)})=(0,0,2)$ by~Proposition~\ref{prop:GENeventypes}, and hence $\gen(T \times H)=*0$.
\end{proof}

\begin{rem}\label{rem:Burnside}
If $d(T \times H)=2$ and $\Phi(T)$ is trivial, then it follows from the Burnside Basis Theorem ~\cite[Theorem~12.2.1]{MarshallHall} that $T$ is isomorphic to either $\mathbb{Z}_2$ or $\mathbb{Z}_2^2$.  
\end{rem}

\begin{lemma}\label{lem:relatinggeneratingsets}
If $T$ a $2$-group and $H$ is a group of odd order, then $\langle S, (t,h) \rangle= \langle S, (t,e), (e,h) \rangle$ for all subsets $S$ of $T \times H$, $t \in T$ of order $2$, and $h \in H$.
\end{lemma}

\begin{proof}
Since $(t,h)=(t,e)(e,h) \in \langle S, (t,h) \rangle$, we have $\langle S, (t,h) \rangle \subseteq \langle S, (t,e), (e,h) \rangle$.  Let $n$ be the order of $h$. Then $(t,e)=(t^n,h^n)=(t,h)^n \in \langle S, (t,h) \rangle$ since $n$ is odd.  We also have $(e,h)=(t,h)^{n+1} \in \langle S, (t,h) \rangle$. Hence $\langle S, (t,h) \rangle \supseteq \langle S, (t,e), (e,h) \rangle$.
\end{proof}

\begin{proposition}
\label{prop:Z2}
If $H$ is a group of odd order and $d(\mathbb{Z}_2 \times H)=2$, then $\gen(\mathbb{Z}_2 \times H)=*0$.
\end{proposition}

\begin{proof}
Since $d(\mathbb{Z}_2 \times H)=2$, we conclude that $d(H)=2$. Let $g:=(x,y) \in \mathbb{Z}_2 \times H$ be the element the first player initially selects, so the game position is $\{g\}\in X_{\lceil g\rceil}\in\mathcal{D}_{\ge 1}$.  If $X_{\lceil g \rceil}\in \mathcal{D}_1$, then the nim-number of $\{g\}$ is clearly not zero so the next player to move, which is the second player, wins. 

If $X_{\lceil g\rceil}\in \mathcal{E}_2$, then the second player selects the identity element of $\mathbb{Z}_2 \times H$ and keeps the resulting game position $\{ g,e \}$ in $X_{\lceil g,e \rceil}=X_{\lceil g\rceil}$.  By Proposition~\ref{prop:GENeventypes}, $\type(X_{\lceil g \rceil})=(0,0,2)$. So the second player wins since the nim-number of $\{ g,e \}$ is $0$. 

It remains to consider the case when $X_{\lceil g\rceil}\in \mathcal{O}_2$, and hence $g=(0,y)$. In this case, the second player picks $(1,e) \in \mathbb{Z}_2 \times H$. We show that the resulting game position $P:=\{(0,y),(1,e)\}$ is in $X_{\lceil P \rceil}\in \mathcal{E}_2$. This will prove that the second player wins since again $P=*0$ by Proposition~\ref{prop:GENeventypes}.

For a contradiction, assume that $X_{\lceil P \rceil}\in \mathcal{E}_1$, so $\langle (0,y),(1,e),(u,v) \rangle=\mathbb{Z}_2 \times H$ for some $(u,v) \in \mathbb{Z}_2 \times H$.  If $u=0$, then by Lemma~\ref{lem:relatinggeneratingsets}, 
\[ 
\mathbb{Z}_2 \times H=\langle (0,y),(1,e),(0,v) \rangle = \langle (0,y),(1,v) \rangle.
\] 
If $u=1$, then we claim that \[\mathbb{Z}_2 \times H=\langle (0,y),(1,e),(1,v) \rangle = \langle (0,y),(1,v) \rangle.\]  Clearly, $\langle (0,y),(1,v) \rangle \subseteq \langle (0,y),(1,e),(1,v) \rangle$,  and $(1,e) \in \langle (0,y),(1,v) \rangle$ by Lemma~\ref{lem:relatinggeneratingsets}, so $\langle (0,y),(1,e),(1,v) \rangle \subseteq \langle (0,y),(1,v) \rangle$.  Thus, the claim holds. In either case, there is an $h\in \mathbb{Z}_2 \times H$ such that $\langle g,h \rangle = \mathbb{Z}_2 \times H$. This implies that $X_{\lceil g \rceil}\in \mathcal{O}_1$, which contradicts the assumption that $X_{\lceil g \rceil} \in \mathcal{O}_2$. Thus, we must have $X_{\lceil P \rceil}\in \mathcal{E}_2$.
\end{proof}

\begin{proposition}
\label{prop:Z2Z2one}
If $H$ is a group of odd order such that $d(H) \leq 1$, then $\gen(\mathbb{Z}_2^2\times H)=*1$.
\end{proposition}

\begin{proof}
Since $d(H)\leq 1$, $\mathbb{Z}_2^2\times H$ is abelian and we conclude that $\gen(\mathbb{Z}_2^2\times H)=*1$ by~\cite[Corollary~8.16]{ErnstSieben}. 
\end{proof}

\begin{figure}
\includegraphics[scale=1.3]{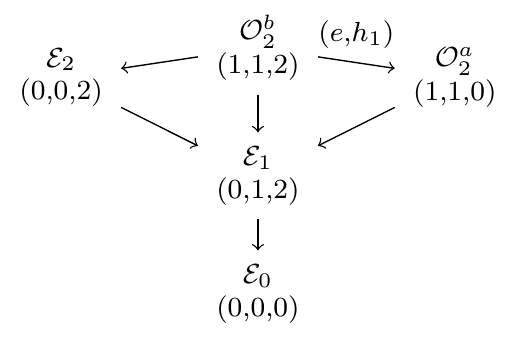}
\caption{\label{fig:MainDiagram}Structure classes for $\gen(\mathbb{Z}_2^2\times H)=*1$ with $d(H)=2$.}
\end{figure}

\begin{proposition}
\label{prop:Z2Z2two}
If $H$ is a group of odd order such that $d(H)=2$, then $\gen(\mathbb{Z}_2^2\times H)=*1$. 
\end{proposition}

\begin{proof}
Let $G=\mathbb{Z}_2^2 \times H$.  We have $d(G)=d(H)=2$ since $\mathbb{Z}_2^2$ and $H$ have coprime orders.  Hence $\mathcal{D}_{\ge 3}=\emptyset$.
Let 
$$
\mathcal{O}_2^a := \{X_I \in \mathcal{O}_2 \mid \opt(X_I)\cap\mathcal{E}_2=\emptyset \},
\quad 
\mathcal{O}_2^b := \mathcal{O}_2 \setminus \mathcal{O}_2^a.
$$
We will show that $\mathcal{O}_1= \emptyset$, and that $\mathcal{E}_m$ for $m \in \{0,1,2\}$, $\mathcal{O}_2^a$, and $\mathcal{O}_2^b$ are nonempty. Then we will use structural induction on the structure classes to show that 
\begin{equation}\label{eqn:types}
	   \type(X_I) = 
	     \begin{dcases}
               (0,0,0), & \text{if } X_I \in \mathcal{E}_0\\
               (0,1,2), & \text{if } X_I \in \mathcal{E}_1\\
               (0,0,2), & \text{if } X_I \in \mathcal{E}_2\\
               (1,1,0), & \text{if } X_I \in \mathcal{O}_2^a \\
               (1,1,2), & \text{if } X_I \in \mathcal{O}_2^b, 
	     \end{dcases}
\end{equation}
as shown in Figure~\ref{fig:MainDiagram}. 

First, we show that $\mathcal{O}_1$ is empty.  Assume $L$ is an intersection subgroup of odd order.  Then $L=\{e\} \times K$ for some subgroup $K$ of $H$.  Since $\delta_{\mathbb{Z}_2^2}(\{e\})=2$, we see that $\delta_G(L) \geq 2$ by Lemma~\ref{lem:FactorDeficiencies}. Hence $X_L \not\in \mathcal{O}_1$, and we conclude that $\mathcal{O}_1=\emptyset$.

Now, we show that $\mathcal{E}_m$ is nonempty for $m \in \{0,1,2\}$.  Let $t$ be a nontrivial element of $\mathbb{Z}_2^2$ and consider $K:=\lceil (t,e) \rceil$, which has even order.  Since $(t,e)$ is contained in the maximal subgroups $\langle t \rangle \times H$ and $\mathbb{Z}_2^2 \times M$ for every maximal subgroup $M$ of $H$, it follows that $K$ is a subgroup of $\langle t \rangle \times \Phi(H)$.  Then \[2 = d(G) \geq \delta_G(K) \geq \delta_G(\langle t \rangle \times \Phi(H)) \geq \delta_H(\Phi(H)) = d(H) = 2,\] by Lemma~\ref{lem:FactorDeficiencies} and~\cite[Corollary~3.3]{BeneshErnstSiebenGeneralizedDihedral}.  Thus, $X_K \in \mathcal{E}_2$.  Since $\mathcal{E}_2$ is nonempty, we can conclude that $\mathcal{E}_1$ and $\mathcal{E}_0$ are nonempty by repeated use of Proposition~\ref{prop:setoptions}.
By Proposition~\ref{prop:GENeventypes}, the types of structure classes in $\mathcal{E}_m$ for $m \in \{0,1,2\}$ are as described in Equation~(\ref{eqn:types}).

We now show that $\mathcal{O}_2^a \not= \emptyset$.   If $u \in \mathbb{Z}_2^2$ is nontrivial with $t \not=u$, then $\langle t \rangle \times H$ and $\langle u \rangle \times H$ are both maximal subgroups of $G$ whose intersection is $\{e\} \times H$.  Hence $\{e\} \times H$ is an intersection subgroup of $G$ with odd order.    Any intersection subgroup $I$ properly containing $\{e\} \times H$ must be isomorphic to $\mathbb{Z}_2 \times H$, so $X_I \in \mathcal{E}_1$ by Lemma~\ref{lem:FullFactors}. Thus $X_{\{e\} \times H} \in \mathcal{O}_2^a$ since $\mathcal{O}_1 = \emptyset$.

Next, we show that $\mathcal{O}_2^b \not= \emptyset$.    By~\cite[Theorem~2]{dlab1960}, \[\Phi(G) = \Phi(\mathbb{Z}_2^2) \times \Phi(H) = \{e\} \times \Phi(H),\] so $\Phi(G)$ has odd order. Hence, $X_{\Phi(G)} \in \mathcal{O}_{d(G)} = \mathcal{O}_2$ by Corollary~\ref{cor:deltaofsubsets}.  Then  \[2 \geq \delta_G(\Phi(G) \cup \{t\}) \geq \delta_G(\mathbb{Z}_2^2 \times \Phi(H)) = \delta_H(\Phi(H)) = d(H) =2\] by Lemma~\ref{lem:FullFactors} and Corollary~\ref{cor:deltaofsubsets}.  So $X_{\lceil \Phi(G),t\rceil} \in \mathcal{E}_2$ by Proposition~\ref{prop:deltaofsubsets}. Thus, $X_{\lceil \Phi(G),t\rceil} \in \mathcal{E}_2$ is an option of $X_{\Phi(G)}$, so $X_{\Phi(G)} \in \mathcal{O}_2^b$.
 
It remains to show that $\type(X_I)=(1,1,0)$ if $X_I \in \mathcal{O}_2^a$ and $\type(X_I)=(1,1,2)$ if $X_I$ is in $\mathcal{O}_2^b$.  If $X_I \in \mathcal{O}_2$, then $X_I$  must have an option in $\mathcal{E}_1$ by Proposition~\ref{prop:setoptions} since $\mathcal{O}_1=\emptyset$, and so $(0,1,2) \in \otype(X_I)$. 

Let $X_I \in \mathcal{O}_2^a$. We first show that $X_I$ has no option in $\mathcal{O}_2^b$.  Suppose toward a contradiction that $X_J \in \mathcal{O}_2^b$ is an option of $X_I$, and let $X_J$ have an option $X_K \in \mathcal{E}_2$.  Let $v \in K$ such that $v$ has order $2$.    Then $\lceil I,v \rceil \leq \lceil J,v \rceil \leq  K$, so $X_I$ has an option $X_{\lceil I,v \rceil} \in \mathcal{E}_2$, which contradicts the  definition of $\mathcal{O}_2^a$.    Thus, $\otype(X_I)$ is either $\{(0,1,2)\}$ or $\{(0,1,2),(1,1,0)\}$ by induction, so $\type(X_I)=(1,1,0)$.

Finally, let $X_I \in \mathcal{O}_2^b$.  Then $X_I$ has an option in $\mathcal{E}_2$ by the definition of $\mathcal{O}_2^b$.  We will show that $X_I$ also has an option in $\mathcal{O}_2^a$.  Let $h_1,h_2 \in H$ such that $H=\langle h_1,h_2\rangle$, and let $J:=\lceil I,(e,h_1)\rceil$. We will show that $X_J \in \mathcal{O}_2^a$ by showing that $X_J \in \mathcal{O}_2$ and $X_{\lceil I, (e,h_1),(s,x)\rceil} \not\in \mathcal{E}_2$ for all $(s,x) \in G$. 
Since $I \cup \{(e,h_1)\}  \subseteq \{e\} \times H$ and $\{e\} \times H$ is an intersection subgroup of odd order, we must have $J \leq \{e\} \times H$. Hence $X_J \in \mathcal{O}$, which implies that $X_J \in \mathcal{O}_2$ since $\mathcal{O}_1 = \emptyset$.  

Now let $(s,x) \in G$.  We will prove that $X_{\lceil J,(s,x)\rceil} \not\in \mathcal{E}_2$.    If $(s,x)$ has odd order, then $s=e$, so $\langle I,(e,h_1),(s,x) \rangle \leq \{e\} \times H$, and thus $X_{\lceil J,(s,x)\rceil} \in \mathcal{O}_2 \not= \mathcal{E}_2$.  Thus, we may assume that $s$ is nontrivial, and we let $w \in \mathbb{Z}_2^2$ be such that $\langle s, w \rangle = \mathbb{Z}_2^2$.  Then \[\langle I, (e,h_1), (s,x), (w,h_2) \rangle =  \langle I,(e,h_1),(e,h_2), (e,x),(s,e),(w,e) \rangle = G\] by two applications of Lemma~\ref{lem:relatinggeneratingsets}, which implies $X_{\lceil I,(e,h_1),(s,x) \rceil} \in \mathcal{E}_1 \not=\mathcal{E}_2$.   Hence $X_J \in \mathcal{O}_2^a$.   Thus, \[\{(0,1,2),(0,0,2),(1,1,0)\} \subseteq \otype(X_I) \subseteq \{(0,1,2),(0,0,2),(1,1,0),(1,1,2)\},\] and so $\type(X_I)=(1,1,2)$.
\end{proof}

The results in this section lead to our main theorem.

\begin{theorem}
\label{mainThm}
If $G= T \times H$ where $T$ is a $2$-group and $H$ is a group of odd order, then
\[
	   \gen(G) = 
	     \begin{dcases}
               *1, & \text{if } |G| \text{ is odd and } d(G) \geq 3\\
               *1, & \text{if } G \cong \mathbb{Z}_{4k} \text{ for some } k\\
               *1, & \text{if } G \cong \mathbb{Z}_2^2 \times H \text{ with } d(H) \leq 2\\
               *2, & \text{if } G \cong \mathbb{Z}_{2} \\
               *2, & \text{if } |G| \text{ is odd and } d(G) \in \{1,2\}\\
               *4, & \text{if } G \cong \mathbb{Z}_{4k+2} \text{ for some } k \geq 1\\
	       *0, & \text{otherwise}.
	     \end{dcases}
\] 
\end{theorem}

\begin{proof}
Each case of the statement follows from an earlier result we proved. The following outline shows the case analysis.  

\begin{myEnumerate}[labelindent=2mm,leftmargin=*]
\item $|G|$ is odd (Proposition~\ref{prop:OrderGIsOdd})
\item $|G|$ is even
\begin{myEnumerate}[labelindent=2mm,leftmargin=*]
  \item $d(G)=1$ (Proposition~\ref{prop:dIs1})
  \item $d(G)\ge 4$ (Proposition~\ref{prop:dIsLarge})
  \item $d(G)= 3$ (Proposition~\ref{prop:dgt3})
  \item $d(G)=2$
    \begin{myEnumerate}[labelindent=2mm,leftmargin=*]
    \item $\Phi(T)$ is nontrivial (Proposition~\ref{PhiTriv})
    \item $\Phi(T)$ is trivial 
      \begin{myEnumerate}[labelindent=0mm,leftmargin=*]
      \item $T\cong \mathbb{Z}_2$ (Proposition~\ref{prop:Z2})
      \item $T\cong \mathbb{Z}_2^2$
        \begin{myEnumerate}[labelindent=3mm,leftmargin=*]
        \item $d(H)\leq 1$ (Proposition~\ref{prop:Z2Z2one})
        \item $d(H)=2$ (Proposition~\ref{prop:Z2Z2two})
        \end{myEnumerate}
      \end{myEnumerate}
    \end{myEnumerate}
\end{myEnumerate}
\end{myEnumerate}
The two cases for when $\Phi(T)$ is trivial are justified by Remark~\ref{rem:Burnside}.
\end{proof}

Recall that every nilpotent group, and hence every abelian group, can be written in the form $T \times H$, where $T$ is a finite $2$-group $T$ and $H$ is a group of odd order. As a consequence, Theorem~\ref{mainThm} provides a complete classification of the possible nim-values for achievement games played on nilpotent groups. Moreover, Theorem~\ref{mainThm} is a generalization of Corollary~8.16 from~\cite{ErnstSieben}, which handles abelian groups only.  Note that even in the case when $H$ is not nilpotent, $H$ must be solvable by the Feit--Thompson Theorem~\cite{FeitThompson}.

\begin{exam} 
The smallest non-nilpotent group that has a Sylow $2$-direct factor is isomorphic to $\mathbb{Z}_2 \times \left(\mathbb{Z}_7 \rtimes \mathbb{Z}_3\right)$, which has order $42$.
\end{exam}

\begin{exam} 
 The smallest group that does not have a Sylow $2$-direct factor is $S_3$.  That is, $S_3$ is the smallest group not covered by Theorem~\ref{mainThm}. However, the possible nim-values for achievement and avoidance games played on symmetric groups were completely classified in~\cite{BeneshErnstSiebenSymAlt}.  The dihedral groups $D_n$ for $n \geq 3$ are not covered by Theorem~\ref{mainThm} either, but these groups were analyzed in~\cite{ErnstSieben}.
\end{exam}


\section{Further Questions}


We mention a few open problems.

\begin{enumerate}
\item What are the nim-numbers of non-nilpotent solvable groups of even order that do not have a Sylow $2$-direct factor?
\item The smallest group $G$ for which $\nim(\gen(G))$ has not been determined by results in ~\cite{BeneshErnstSiebenSymAlt,BeneshErnstSiebenDNG,BeneshErnstSiebenGeneralizedDihedral,ErnstSieben} or Theorem~\ref{mainThm} is the dicyclic group $\mathbb{Z}_3 \rtimes \mathbb{Z}_4$.  All dicyclic groups have Frattini subgroups of even order. Hence these groups have nim-number $0$ as a consequence of Proposition~\ref{prop:GENeventypes}.  The smallest group not covered in the current literature is $\mathbb{Z}_3 \times S_3$.  What are the nim-numbers for groups of the form $\mathbb{Z}_m \times S_n$ for $m \geq 2$ and $n \geq 3$? 

\item The nim-numbers of some families of nonsolvable groups were determined in~\cite{BeneshErnstSiebenSymAlt}.  Can we determine the nim-numbers for all nonsolvable groups?
\end{enumerate}

\bibliographystyle{plain}
\bibliography{game}

\end{document}